\documentclass[%
reprint%
,onecolumn%
,jmp%
,a4paper,superscriptaddress,nofootinbib,showkeys]{revtex4-2}

\usepackage{mathrsfs}

\usepackage{amsmath,amssymb,bm,cancel}
\usepackage{amsthm}
\usepackage{mathtools}

\newtheorem{theorem}{Theorem}[section]
\newtheorem{proposition}{Proposition}[section]
\newtheorem{lemma}{Lemma}[section]
\newtheorem{corollary}{Corollary}[section]

\theoremstyle{definition}

\theoremstyle{remark}
\newtheorem{remark}{Remark}[section]

\usepackage{graphicx}
\usepackage[dvipsnames]{xcolor}
\usepackage{tikz}
\usetikzlibrary{arrows.meta}

\DeclareMathOperator{\interior}{int}

\begin{document}


\setlength\parindent{0pt}
\setlength{\parskip}{0.5em}

\title{A Stefan-Sussmann theorem for normal distributions on manifolds with boundary}

\author{David Perrella}
\author{David Pfefferl{\'e}}
\author{Luchezar Stoyanov}
\affiliation{The University of Western Australia, 35 Stirling Highway, Crawley WA 6009, Australia}

\maketitle

\section{Abstract}
An analogue of the Stefan-Sussmann Theorem on manifolds with boundary is proven for normal distributions. These distributions contain vectors transverse to the boundary along its entirety. Plain integral manifolds are not enough to ``integrate" a normal distribution; the next best ``integrals" are so-called neat integral manifolds with boundary. The conditions on the distribution for this integrability is expressed in terms of adapted collars and integrability of a pulled-back distribution on the interior and on the boundary.

\section{Introduction}

The foundations of integrable distribution theory on manifolds can be summarised by the following four programs.
\begin{enumerate}
    \item Define integrability of a distribution.
    \item Ask what ``differential" conditions are equivalent to integrability.
    \item Determine the ``regularity" of integral manifolds. 
    \item Look at a notion of ``maximal" for integral manifolds and show that there is a correspondence with foliations.
\end{enumerate}

In the case of constant rank, the Frobenius Theorem addresses program two and is the main component of the foundations (although, as noted by Lavau \cite{Lavau(expositor)}, the precise historical list of contributors to the foundations of this theory do include authors before Frobenius). In the case of non-constant rank, it was the independent works of Stefan \cite{Stefan} and Sussmann \cite{Sussmann} which gave rise to the completion of the foundations; culminating in the Stefan-Sussmann Theorem (which again addresses program two). What appears absent from the literature is an analogous theory on manifolds with boundary, where the integral manifolds become, in general, manifolds with boundary.

To address this, the present paper takes on the analogous first and second programs for \emph{normal distributions} on manifolds with boundary. Normal distributions are those which contain vectors transverse to the boundary along its entirety. Our analogue of integrability is what we call \emph{neat integrability} and the differential conditions constitute our main result, namely a Stefan-Sussmann theorem for normal distributions.

In a future paper, we will use the results obtained here to complete the remaining two analogous programs in the foundations. This will provide means of guaranteeing examples of \emph{foliations by manifolds with boundary} on manifolds with boundary. A general notion of these foliations were considered by Noakes \cite{Noakes} where he generalised Bott's Theorem \cite{Bott} as well as demonstrate how these foliations may be used to guarantee (ordinary) foliations.

This paper is structured as follows. Distributions and the Stefan-Sussmann Theorem are first reviewed in section~\ref{sec:review_stefan-Sussmann}. We then discuss normal distributions and state our main result in section \ref{sec:norm_dists_and_main_result}. Lastly, we give a proof of this result in section \ref{sec:Proof}.

\section{Review of distributions and the Stefan-Sussmann Theorem}
\label{sec:review_stefan-Sussmann}
The definitions of the standard terms we use from smooth manifold theory coincide with that of Lee's book \cite{Lee} and of Rudolf and Schmidt's book \cite{RudolphSchmidt}. In particular, a submanifold with or without boundary in general is not assumed to be embedded, only immersed. We compile the relevant theory of distributions from Rudolf and Schmidt's book \cite{RudolphSchmidt} using our own terminology. Let $M$ be a smooth manifold.

For an arbitrary subset $D \subset TM$, a \textbf{\emph{$\bm{D}$-section}} is a vector field $X$ on $M$ such that for all $p \in M$, $X|_p \in D$. We also denote $D_p = D \cap T_pM$ for $p \in M$. With this, $D$ is called a \textbf{\emph{distribution on $\bm{M}$}} if for all $p \in M$, $D_p$ is a vector subspace of $T_pM$ and for any $v \in D_p$, there exists a $D$-section $X$ with $X|_p = v$. The \textbf{\emph{rank of $\bm{D}$}} is the map $r : M \to \mathbb{N}_0$ given by $r(p) = \dim D_p$. For a smooth map $f : S \to M$ between manifolds, we set $f^*D = (Tf)^{-1}(D)$ and call $f^*D$ the \textbf{\emph{pullback of $\bm{D}$ on to $\bm{S}$}} (where $Tf$ denotes the tangent map).

For the following, let $D$ be a distribution on $M$ of rank $r$. An \textbf{\emph{integral manifold}} of $D$ is a connected submanifold $I$ of $M$ such that, for $p \in I$, $Ti|_p(T_pI) = D_p$ ($i : I \subset M$). We say that $D$ is \textbf{\emph{integrable}} if through every point of $M$, there passes an integral manifold of $D$.

We now introduce the definitions needed to state the Stefan-Sussmann Theorem. To this end, a \textbf{\emph{local $D$-section}} is a local section $X : U \to TM$ (where $U$ is an open submaifold of $M$) with $X|_p \in D$ for $p \in U$. A distribution $D$ on $M$ is said to be:
\begin{enumerate}
    \item \textbf{\emph{involutive}} if its $D$-sections form a Lie subalgebra of the space of vector fields on $M$,
    \item \textbf{\emph{homogeneous}} (after Stefan \cite{Stefan}) if for all local $D$-sections $X$, for any point $(p,t) \in \mathcal{D}$ in the (maximal) domain of the flow $\psi^X : \mathcal{D} \to M$, one has (after identifying tangent spaces) $T\psi^X_t|_p(D_p) = D_{\psi^X(p,t)}$. 
\end{enumerate} 

The final definition pertains to certain nice charts for $D$. Let $p \in M$ and set $n = \dim M$. A chart $(U,\varphi)$ on $M$ is said to be \textbf{\emph{adapted to $\bm{D}$ at $\bm{p}$}} if the following holds (where $(\partial_1,...,\partial_n)$ denotes the local frame of $TM$ induced by $\varphi$).
\begin{enumerate}
    \item $\varphi(p) = 0$ and $\varphi(U) = (-\epsilon,\epsilon)^n$ for some $\epsilon > 0$,
    \item $\partial_1,...,\partial_{r(p)}$ are local $D$-sections,
    \item for all $c \in (-\epsilon,\epsilon)^{n-r(p)}$, denoting $U_c = \varphi^{-1}\left((-\epsilon,\epsilon)^{r(p)}\times \{c\}\right)$ if $r(p) < n$ and $U_c = U$ otherwise, the rank $r$ of $D$ is constant along $U_c$.
\end{enumerate}

We now state the Stefan-Sussmann Theorem (as formulated in Rudolf and Schmidt's book \cite{RudolphSchmidt}).

\begin{theorem}[Stefan-Sussmann Theorem]\label{Stefan-Sussmann}
Let $M$ be a manifold (without boundary) and $D$ a distribution on $M$. The following are equivalent.
\begin{enumerate}
    \item $D$ is integrable.
    \item $D$ is involutive and has constant rank along the integral curves of local $D$-sections.
    \item $D$ is homogenous.
    \item For every $p \in M$, there exists a chart adapted to $D$ at $p$.
\end{enumerate}
\end{theorem}

The purpose of this paper is to prove an analogous result for what we are calling \emph{normal distributions on manifolds with boundary}. We now discuss this notion and state our main result.

\section{Normal distributions and the main result}
\label{sec:norm_dists_and_main_result}
The definition of distribution is precisely as before except on a manifold with boundary (including the notions of rank, pullback and involutivity). For the remainder of this section, let $D$ be a distribution on a manifold with boundary $M$.

A connected submanifold $L$ with boundary of $M$ is called an \textbf{\emph{integral manifold with boundary of $\bm{D}$}} if for $p \in L$, $Tl|_p(T_pL) = D_p$ ($l : L \subset M$). We will say that $D$ is \textbf{\emph{integrable by boundaries}} if through every point of $M$, there passes an integral manifold with boundary of $D$.

The distribution $D$ is said to be \textbf{\emph{normal}} if for every $p \in \partial M$, there exists a vector in $D_p$ transverse to $\partial M$. To introduce our analogue of the Stefan-Sussmann Theorem (Theorem \ref{Stefan-Sussmann}), we introduce some definitions. For the following, let $D$ be a normal distribution on $M$ of rank $r$.

An integral manifold with boundary $L$ of $D$ is said to be \textbf{\emph{neat}} if $\partial L \subset \partial M$. Accordingly, $D$ is said to be \textbf{\emph{neatly integrable by boundaries}} if through any point in $M$, there passes a neat integral manifold of $D$. 

We now move to analogues of the hypotheses within the Stefan-Sussmann theorem. For the remaining of this section, we assume that $\partial M \neq \emptyset$. The analogues we present will be in terms of some collar; without mentioning any general vector field. This simplification is made by the hypotheses expressing integrability of some constructed distributions on the manifolds $\interior M$ and $\partial M$. One can express our main result purely in terms of vector fields by applying the (standard) Stefan-Sussmann Theorem (Theorem \ref{Stefan-Sussmann}) and the Boundary Flowout Theorem (Theorem \ref{Boundary Flowout Theorem}).

Recall, a \textbf{\emph{collar}} $\Sigma : \partial M \times [0,1) \to M$ is an embedding onto a neighbourhood of $\partial M$ in $M$ such that $\Sigma(p,0) = p$ for all $p \in \partial M$. We will call a collar $\Sigma : \partial M \times [0,1) \to M$ a \textbf{\emph{$\bm{D}$-precollar}} if, for all $(p,t) \in \partial M \times [0,1)$,
\begin{equation*}
r(\Sigma(p,t)) = r(p) \text{ and } \Sigma(p,\cdot )'(t) \in D_{\Sigma(p,t)}. 
\end{equation*}
The normal distribution $D$ will be called \textbf{\emph{precollared}} if a $D$-precollar exists. A collar $\Sigma : \partial M \times [0,1) \to M$ will called a \textbf{\emph{$D$-collar}} if, denoting $\jmath : \partial M \subset M$, for all $(p,t) \in \partial M \times [0,1)$,
\begin{equation*}
T\Sigma(\cdot,t)|_p((\jmath^*D)_p) + T\Sigma(p,\cdot )|_t(T_t[0,1)) = D_{\Sigma(p,t)}.
\end{equation*}
Accordingly, the normal distribution $D$ will be called \textbf{\emph{collared}} if a $D$-collar exists.

Our analogue of adapted charts is closest to the original. However, to ensure these charts account for the distribution being normal, we must use the local vector field generated by the last coordinate of the chart; since this is the only local vector field which is transverse to the boundary. We have arranged this by using the last few coordinates of the chart as opposed to the first few.

Let $p \in M$ and set $n = \dim M$. Then, a chart $(U,\varphi)$ is said to be \textbf{\emph{neatly adapted to $\bm{D}$ at $\bm{p}$}} if the following holds.
\begin{enumerate}
    \item $\varphi(p) = 0$ and $\varphi(U) = (-\epsilon,\epsilon)^n$ if $p \in \interior M$ and $\varphi(U) = (-\epsilon,\epsilon)^n \cap \mathbb{H}^n$ otherwise,
    \item $\partial_{n-r(p)+1},...,\partial_n$ are local $D$-sections,
    \item For all $c \in (-\epsilon,\epsilon)^{n-r(p)}$, denoting $U_c = \varphi^{-1}\left(\{c\} \times (-\epsilon,\epsilon)^{r(p)}\right)$ if $r(p) < n$ and $U_c = U$ otherwise, the rank $r$ of $D$ is constant along $U_c$.
\end{enumerate}

With these definitions, we may present our main result.

\begin{theorem}\label{Normal Stefan-Sussmann}
Let $M$ be a manifold with boundary $\partial M \neq \emptyset$ and $D$ be a normal distribution of rank $r$ on $M$. Denote by $\imath : \interior M \subset M$ and $\jmath : \partial M \subset M$ the inclusions of the interior and the boundary into $M$. Then, the pullback $\imath^*D$ is a distribution of rank $r|_{\interior M}$ on $\interior M$ and $\jmath^*D$ is a distribution of rank $r|_{\partial M}-1$ on $\partial M$. Moreover, the following are equivalent.
\begin{enumerate}
    \item[S1.] $D$ is neatly integrable by boundaries.
    \item[S2.] $\imath^*D$ is integrable and $D$ is precollared.
    \item[S3.] $\imath^*D$ and $\jmath^*D$ are integrable and $D$ is collared.
    \item[S4.] For every $p \in M$, there exists a chart neatly adapted to $D$ at $p$.
\end{enumerate}
\end{theorem}

Using Theorem \ref{Normal Stefan-Sussmann} and Proposition \ref{preprecollar} (found in the next section), a Frobenius theorem may be formulated as follows.

\begin{corollary}\label{Normal Frobenius}
Let $M$ be a manifold with boundary $\partial M \neq \emptyset$ and $D$ be a normal distribution of constant rank. Then the following are equivalent.
\begin{enumerate}
    \item[F1.] $D$ is neatly integrable by boundaries.
    \item[F2.] $D$ is involutive.
    \item[F3.] For every $p \in M$, there exists a chart neatly adapted to $D$ at $p$.
\end{enumerate}
\end{corollary}

We will now prove Theorem \ref{Normal Stefan-Sussmann}.

\section{Proof of main result}
\label{sec:Proof}
Our proof will establish five implications directly. Although this is not the optimal amount (four is optimal), three of our chosen implications are trivial, making for only two substantial implications. We will now establish some preliminary observations of distributions which we used throughout in proving these implications.

\subsection{Preliminary observations of distributions and of normal distributions}

We first state two additional results found in Rudolf and Schmidt's book \cite{RudolphSchmidt} for later use. The first is on the ``regularity" of integral manifolds of distributions. They are in fact \emph{weakly embedded}. We introduce some notation to express this and related notions to simplify subsequent proofs.

Let $A,B,M,N$ be manifolds with boundary and $f : M \to N$ be a map between manifolds with boundary. 
\begin{enumerate}
    \item If $A \subset M$, the restriction is denoted $f|_A : A \to N$ as usual.
    \item If $f(M) \subset B$, we will write the induced map as $f|^B : M \to B$. We call $f|^B$ the \textbf{\emph{co-restriction of $\bm{f}$ to $B$}} if $B \subset N$ and call it the \textbf{\emph{co-extension of $\bm{f}$ to $\bm{B}$}} if $B \supset Y$. If $f(A) \subset B$, we define $f|_A^B = (f|_A)^B$ for succinctness.
    \item If $A$ is an submanifold with boundary in $M$ and $f$ is an injective immersion and $B$ is the unique submanifold with boundary of $N$ such that $f|_A^B$ is a diffeomorphism, we denote $f(A) = B$.
    \item If $M$ is a manifold and $A$ is a submanifold of $M$, it is said that \textbf{\emph{$\bm{A}$ is weakly embedded in $\bm{M}$}} if for any smooth map $F : S \to M$ between manifolds such that $F(S) \subset A$, we have that the co-restriction $F|^A$ is also smooth.
    \item We will denote by $\text{Pts}(A)$ the set of points of $A$ when it is necessary for distinction.
\end{enumerate}

The following is Proposition 3.5.15 in Rudolf and Schmidt's book \cite{RudolphSchmidt}.

\begin{proposition}\label{boundarylessweakembedded}
Let $M$ be a manifold and $D$ an integrable distribution on $M$. Then, all integral manifolds of $D$ are weakly embedded in $M$.
\end{proposition}

Another important result found in Rudolf and Schmidt's book \cite{RudolphSchmidt} is on the existence of maximal integral manifolds. An integral manifold $I$ of a distribution $D$ on a manifold $M$ is \textbf{\emph{maximal}} if, for any integral manifold $H$ of $D$ with a point in common with $I$, $H$ is an open submanifold of $I$. The following is Theorem 3.5.17 in Rudolf and Schmidt's book \cite{RudolphSchmidt} adapted to our context.

\begin{theorem}\label{boundarylessmaximal}
Let $M$ be a manifold and $D$ an integrable distribution on $M$. Then, through every point of $M$, there passes a unique maximal integral manifold of $D$.
\end{theorem}

This gives rise to a convenient containment property of integral curves.

\begin{proposition}\label{max contain int curves}
Let $M$ be a manifold and $D$ an integrable distribution on $M$. Let $X$ be a $D$-section, $\gamma : (-\epsilon,\epsilon) \to M$ be an integral curve of $X$ and $I$ be a maximal integral manifold of $D$ with $\gamma \cap I \neq \emptyset$. Then, $\gamma \subset I$.
\end{proposition}

\begin{proof}
First let $H$ be a maximal integral manifold of $D$. We see that $X$ is tangent to $H$. Hence, there is a $h$-related vector field $\tilde{X}$ to $X$ where $h : H \subset M$ is the inclusion. Then consider $\gamma^{-1}(H)$. Let $t \in \gamma^{-1}(H)$. Then considering the integral curves of $\tilde{X}$, by uniqueness, there is a neighbourhood $(t-\delta,t+\delta)$ of $t$ for some $\delta > 0$ with $(t-\delta,t+\delta) \subset \gamma^{-1}(I)$. Hence, $\gamma^{-1}(I)$ is open in the interval $(-\epsilon,\epsilon)$.

By assumption $\gamma^{-1}(I) \neq \emptyset$. Hence, by the above, $\gamma^{-1}(I)$ is open and nonempty in $(-\epsilon,\epsilon)$. Now consider the compliment $(-\epsilon,\epsilon) \backslash \gamma^{-1}(I) = \gamma^{-1}(M \backslash I)$. By Theorem \ref{boundarylessmaximal}, replacing each $p \in M$ with the unique maximal integral manifold $L$ of $D$ containing $p$, we obtain a partition $\mathcal{C}$ of $M$ by the maximal integral manifolds of $D$. Hence, $M\backslash I = \cup_{H \in \mathcal{C}\backslash \{I\}} H$ so that,
\begin{equation*}
(-\epsilon,\epsilon) \backslash \gamma^{-1}(I) = \gamma^{-1}(M \backslash I) = \gamma^{-1}(\cup_{H \in \mathcal{C}\backslash \{I\}} H)  = \cup_{H \in \mathcal{C}\backslash \{I\}} \gamma^{-1}(H).
\end{equation*}
Hence, using the above again, $\gamma^{-1}(M \backslash I)$ is open in $(-\epsilon,\epsilon)$. Since $(-\epsilon,\epsilon)$ is connected, it follows that $\gamma^{-1}(I) = (-\epsilon,\epsilon)$. That is, $\gamma \subset I$.
\end{proof}

We now make some observations which focus on normal distributions. These will enable the use of collars with normal distributions as well as prove the opening statement of Theorem \ref{Normal Stefan-Sussmann}. We first state two results found in Lee's book for later use.

Let $M$ be a manifold with boundary. Let $v \in TM$ be a vector on the boundary $\partial M$. If $v$ is transverse to $\partial M$, following Lee \cite{Lee}, $v$ is said to be:
\begin{enumerate}
    \item \textbf{\emph{inward-pointing}} if there exists a curve $\gamma : [0,\epsilon) \to M$ for some $\epsilon > 0$ such that $\gamma'(0) = v$;
    \item \textbf{\emph{outward-pointing}} if there exists a curve $\gamma : (-\epsilon,0] \to M$ for some $\epsilon > 0$ such that $\gamma'(0) = v$.
\end{enumerate}

We will make use of the following, which is Proposition 5.41 from Lee's book \cite{Lee} and a useful corollary.

\begin{proposition}\label{trichotomy}
For $p \in \partial M$, $T_pM$ is partitioned by the set of tangent vectors, the set of inward-pointing vectors and the set of outward-pointing vectors.
\end{proposition}

\begin{corollary}
A vector $v \in TM$ on the boundary is tangent to $\partial M$ if and only if there exists a curve $\gamma : (-\epsilon,\epsilon) \to M$ for some $\epsilon > 0$ such that $\gamma'(0) = v$.
\end{corollary}

For collar constructions, we will frequently use the Boundary Flowout Theorem (Theorem 9.24 in Lee's book \cite{Lee}).

\begin{theorem}[Boundary Flowout Theorem]\label{Boundary Flowout Theorem}
Let $M$ be a manifold with boundary $\partial M \neq \emptyset$ and $N$ an inward-pointing vector field on $M$. Then, there exists a smooth $\delta : \partial M \to \mathbb{R}^+$ and embedding $\Phi : U \to M$ where $U$ is an open submanifold with boundary in $\partial M \times [0,\infty)$ given by,
\begin{equation*}
U = \{ (p,t) \in \partial M \times [0,\infty) : t < \delta(p)\},
\end{equation*}
such that $\Phi(p,\cdot) : [0,\delta(p)) \to M$ is an integral curve of $N$ starting at $p$ for $p \in \partial M$.
\end{theorem}

For the remaining of this section, let $M$ be a manifold with boundary and $D$ be a normal distribution on $M$ with rank $r$. 

We first note that boundaries are essential for integrability of normal distributions by boundaries.

\begin{proposition}\label{interior}
If $K$ is an integral manifold with boundary of the normal distribution $D$, then $\interior K \cap \partial M = \emptyset$.
\end{proposition}

\begin{proof}
Assume there is a $p \in \interior K \cap \partial M$. Taking a vector $N \in D_p$ transverse to $\partial M$, we get an $i$-related vector $n$ on $\interior K$ ($i : \interior K \subset M$) and hence a curve $\gamma : (-\epsilon,\epsilon) \to \interior K$ with $\gamma'(0) = n$. However, then $\gamma|^M: (-\epsilon,\epsilon) \to M$ is with $(\gamma|^M)'(0) = N$, contradicting $N$ being transverse to $\partial M$ from Proposition \ref{trichotomy}.
\end{proof}

\begin{proposition}\label{globalinward}
There exists an inward-pointing vector field $N$ on $M$ which is also a $D$-section.
\end{proposition}

\begin{proof}
For each $p \in \partial M$, we may choose an open neighbourhood $U_p$ in $M$ containing $p$ and a $D$-section $N_p$ such that $N_p$ is inward-pointing on $U_p \cap \partial M$. Then, for the cover $\mathcal{U} = \{ U_p : p \in \partial M \} \cup \{ \interior M \}$ of $M$, take a partition of unity $\mathcal{P} = \{\rho_p : p \in \partial M\} \cup \{\rho\}$ on $M$ subordinated to $\mathcal{U}$. Then, the vector field $N = \sum_{p \in \partial M}\rho_p N_p$ on $M$ is inward-pointing and is a $D$-section.
\end{proof}

The Boundary Flowout Theorem (Theorem \ref{Boundary Flowout Theorem}) enables the use of collars with $D$ as follows.

\begin{proposition}\label{preprecollar}
There is a collar $\Sigma : \partial M \times [0,1) \to M$ whose diffeomorphism co-restriction $\sigma = \Sigma|^V$ (where $V = \Sigma(\partial M \times [0,1))$) satisfies the following. There is an inward-pointing $i^*D$-section $\tilde{N}$ ($i : V \subset M$) such that for each $p \in \partial V = \partial M$, the map $\sigma(p,\cdot) : [0,1) \to V$ is an integral curve starting at $p$.
\end{proposition}

\begin{proof}
Given such an $N$ from Proposition \ref{globalinward} and the resulting embedding $\Phi : U \to M$ from Theorem \ref{Boundary Flowout Theorem}, as seen in the proof of the Collar neighbourhood Theorem (Theorem 9.25 in Lee's book \cite{Lee}) we get an embedding $\Sigma : \partial M \times [0,1) \to M$ given by $\Sigma(p,t) = \Phi(p,\delta(p)t)$. The map $\sigma = \Sigma|^V$ trivially inherits the desired properties from $\Phi$.
\end{proof}

We will also require some simple pull-back properties for distributions.

\begin{proposition}\label{pullback dist to codim 1}
Let $M$ be a manifold with boundary and $D$ a distribution on $M$. Let $S$ be a manifold with boundary of codimension $1$ and $F : S \to M$ an injective immersion. Let $\omega$ be a $1$-form on $M$ such that $F^*\omega = 0$ and $X$ be a $D$-section with $\omega(X) = 1$. Then, $F^*D$ is a distribution on $S$ of rank $r \circ F-1$ where $r$ is the rank of $D$.
\end{proposition}

\begin{proof}
Let $p \in S$ and denote $q = F(p)$. From $\dim{S} = \dim{M}-1$, $\omega|_q(X|_q) = 1$, and injectivity of $TF|_p:T_pS\to T_qM$, we have that $\mathop{\text{im}}{TF|_p} = \ker \omega|_q$. Now, consider the linear map $A_q : T_q M \to T_q M$ given by $A_q(u) = u - \omega|_q(u)X|_q$. We claim that $TF|_p((F^*D)_p) = A_q(D_q)$.

Indeed, one readily sees that $(F^*D)_p = (TF|_p)^{-1}(D_q)$ and $\ker \omega|_q = \mathop{\text{im}}{A_q}$ so that,
\begin{equation*}
TF|_p((F^*D)_p) = TF|_p((TF|_p)^{-1}(D_q)) = \mathop{\text{im}}{TF|_p} \cap D_q = \ker \omega|_q \cap D_q = \mathop{\text{im}}{A_q} \cap D_q.
\end{equation*}
We now consider $\mathop{\text{im}}{A_q} \cap D_q$. For any $u,v \in T_qM$ such that $v = A_q(u)$, we have $v = u - \omega|_q(u)X|_q$. Hence, since $X|_q \in D_q$ and $D_q$ is a vector subspace of $T_qM$, we have $u \in D_q$ if and only if $v \in D_q$. This shows that $\mathop{\text{im}}{A_q} \cap D_q = A_q(D_q)$. Hence, in total,
\begin{equation*}
TF|_p((F^*D)_p) = \mathop{\text{im}}{A_q} \cap D_q = A_q(D_q),
\end{equation*}
as claimed.

To compute the dimension of $(F^*D)_p$, we consider the linear map $B_q = A_q|_{D_q}$. Since $\dim{\ker {B_q}} = 1$, the vector subspace $TF|_p((F^*D)_p) = A_q(D_q)$ has dimension $\dim \text{im} B_q = r(q)-1$. Hence, because $TF|_p$ is a linear injection, $(F^*D)_p$ is a vector subspace of $T_pS$ of dimension $r(q)-1 = (r \circ F-1)(p)$.

Now, let $v \in (F^*D)_p$. Then, $TF(v) \in D$ and $v \in T_pS$ so that $w = TF|_p(v) \in D_q$. Hence, there is a $D$-section $W$ with $W|_q = w$. Consider the vector field $A(W) = W - \omega(W)X$. For $p' \in S$, we have that,
\begin{equation*}
A(W)|_{F(p')} = A_{F(p')}(W|_{F(p')}) \in A_{F(p')}(D_{F(p')}) =  TF|_{p'}((F^*D)_{p'}),     
\end{equation*} 
so that $W|_{F(p')} = TF|_{p'}(u)$ for some $u \in (F^*D)_{p'} \subset T_{p'}S$. Hence, since $F$ is an immersion, there exists a unique map $Y : S \to TS$ with $Y|_{p'} \in T_{p'}S$ for $p' \in S$ such that $A(W)\circ F = TF \circ Y$. Then, since $F(S)$ is an submanifold with boundary in $M$, using a trivial extension of Proposition 8.23 to the case of $M$ with boundary from Lee's book \cite{Lee} (on restriction of vector fields to submanifolds with boundary), we obtain that $Y$ is a vector field on $S$. It is then clear that $Y$ is a $F^*D$-section with $Y|_p = v$.
\end{proof}

\begin{remark}\label{localproperty}
Let $M$ be a manifold with boundary and $D$ be a subset of $TM$. If $D$ is a distribution of rank $r$ on $M$, then for any open submanifold $U$ with boundary, denoting $\iota : U \subset M$, the subset $i*D$ is a distribution of rank $r|_{U}$ on $U$. Using smooth bump functions, the converse can also be established. Namely, if every point $p \in M$ has a neighbourhood $U$ for which $i^*D$ is a distribution on $U$ (where $\iota : U \subset M$), then $D$ is a distribution on $M$.
\end{remark}

With these pull-back properties we may in particular prove the opening statement of Theorem \ref{Normal Stefan-Sussmann}.

\begin{lemma}\label{boundary dist}
The pullback $\jmath^*D$ ($\jmath : \partial M \subset M$) is a distribution on $\partial M$ of rank $r|_{\partial M} - 1$.
\end{lemma}

\begin{proof}
Take a collar $\Sigma$ as in Proposition \ref{preprecollar} and adopt the notation from there. We have that $i^*D$ is a distribution of rank $r|_{V}$ whereby $\tilde{N}$ is a $i^*D$-section (see Remark \ref{localproperty}). With the diffeomorphism $\sigma$, we have a smooth function $f = (\pi_2 \circ \sigma^{-1})|^{\mathbb{R}}$ where $\pi_2 : \partial M \times [0,1) \to [0,1)$ is the projection onto $[0,1)$. 

Then, note that $\partial V = \partial M$ so that $\partial M$ is an embedded submanifold of $V$ where the inclusion $\jmath' : \partial M \subset V$ satisfies $i \circ \jmath' = \jmath$. Setting $\omega = df$, $\omega$ is a 1-form on $V$ we have $\jmath'^*\omega = 0$ and $\omega(\tilde{N}) = 1$. Hence, by Proposition \ref{pullback dist to codim 1}, $\jmath^*D = (i \circ \jmath')^*D = \jmath'^*(i^*D)$ is a distribution on $\partial M$ of rank $r|_{V}|_{\partial M}-1 = r|_{\partial M}-1$.
\end{proof}

Lemma \ref{boundary dist}, combined with Remark \ref{localproperty}, allows us to record the following; which is the opening statement of Theorem \ref{Normal Stefan-Sussmann}.

\begin{corollary}\label{opening of Normal Steffan-Sussmann}
Let $M$ be a manifold with boundary and $D$ be a normal distribution of rank $r$ on $M$. Then, the pullback $\imath^*D$ ($\imath : \interior M \subset M$) is a distribution of rank $r|_{\interior M}$ on $\interior M$ and $\jmath^*D$ ($\jmath : \partial M \subset M$) is a distribution of rank $r|_{\partial M}-1$ on $\partial M$.
\end{corollary}

\subsection{Implication S1 $\Rightarrow$ S2}
Let $M$ be a manifold with boundary $\partial M \neq \emptyset$ and $D$ be a normal distribution of rank $r$ on $M$. Assume that $D$ is neatly integrable. Let $\imath : \interior M \subset M$ and $\jmath : \partial M \subset M$ denote the inclusions.

\begin{lemma}
The pullback distribution $\imath^*D$ is integrable.
\end{lemma}

\begin{proof}
Let $p \in \interior M$. Take a neat integral manifold $L$ of $D$ with $p \in L$. Then, we have $p \in \interior L$. Since $L$ is connected, $\interior L$ is connected also. Since $\interior L \subset \interior M$, we have that $\interior L$ is immersed in $\interior M$. Clearly, $Ti|_q(T_q\interior L) = (\imath^*D)_q$ for $q \in \interior L$ (where $i : \interior L \subset \interior M$).
\end{proof}

\begin{lemma}\label{boundaryint}
Let $L$ be a neat integral manifold with boundary of $D$. Then, $Tj|_p(T_p\partial L) = (\jmath^*D)_p$ for $p \in \partial L$ where $j : \partial L \subset \partial M$.
\end{lemma}

\begin{proof}
Let $p \in \partial L$. Then $\partial L \neq \emptyset$ and denoting $l : L \subset M$ and $g : \partial L \subset L$, we get,
\begin{equation*}
\begin{split}
T\jmath|_p(Tj|_p(T_p\partial L)) &= T(\jmath \circ j)|_p(T_p\partial L)\\
&= T(l \circ g)|_p(T_p\partial L)\\
&= Tl|_p(Tg|_p(T_p\partial L))\\
&\subset Tl|_p(T_pL)\\
&= D_p.
\end{split}
\end{equation*}

Hence, $Tj|_p(T_p\partial L) \subset (T\jmath|_p)^{-1}(D_p) = (\jmath^*D)_p$. Since $\dim{T_pL} = \dim{D_p}$, we have by Proposition \ref{boundary dist} that $\dim{Tj|_p(T_p\partial L)} =\dim{ (\jmath^*D)_p}$. Hence, we must have that $Tj|_p(T_p\partial L) = (\jmath^*D)_p$.
\end{proof}

\begin{lemma}
The pullback $\jmath^*D$ is integrable.
\end{lemma}

\begin{proof}
Let $p \in \partial M$. Again, take a neat integral manifold with boundary $L$ of $D$ with $p \in L$. Then $p \in \partial L$ from Proposition \ref{interior}. Let $J$ be the connected component of $\partial L$ containing $p$ (an open submanifold of $\partial L$). By Lemma \ref{boundaryint}, $J$ is an integral manifold of $\jmath^*D$.
\end{proof}

\begin{lemma}\label{interiorexpand}
For every $p \in \partial M$, there exists a neat integral manifold $L$ of $D$ with boundary such that $p \in L$ and $\interior L$ is a maximal integral manifold of $\imath^*D$.
\end{lemma}

\begin{proof}
Let $p \in \partial M$. By assumption, there exists a neat integral manifold $K$ of $D$ with boundary such that $p \in K$. Then, $\interior K$ is an integral manifold of $\imath^*D$. Then, from Theorem \ref{boundarylessmaximal}, there exists a maximal integral manifold $I$ of $\imath^*D$ whereby $\interior K$ is an open submanifold of $I$. Since also $\interior K$ is an open submanifold of $K$ and $I \cap K = \text{Pts}(\interior K)$, we have a manifold with boundary $L$ with $\text{Pts}(L) = I \cup K$ (so $p \in L$) such that $I$ is an open submanifold of $L$ and $K$ is an open submanifold with boundary in $L$. It follows easily that $L$ is an integral manifold with boundary of $D$. Moreover, since again $\interior K$ is an open submanifold of $I$, we obtain that $\interior L = I$ and $\partial L = \partial K \subset \partial M$ (so $L$ is neat).
\end{proof}

For the following, fix a collar $\Sigma : \partial M \times [0,1) \to M$ as in Proposition \ref{preprecollar} and inherit the surrounding notation from there.

\begin{lemma}\label{intecollar}
Let $p \in \partial M$. Then, there exists a neat integral manifold $L$ containing $p$ and a collar $\rho$ on $L$ such that $l \circ \rho = \Sigma \circ k$.
\end{lemma}

\begin{proof}
Fix an integral manifold $L$ of $D$ with boundary such that $p \in L$ as in Lemma \ref{interiorexpand}. Denote by $k : \partial L \times [0,1) \subset \partial M \times [0,1)$ and $l : L \subset M$ the inclusions.

Consider the diffeomorphism $\sigma = \Sigma|^V : \partial M \times [0,1) \to V$, the inclusion $h_1 : V \cap L \subset V$ and the $h_1$-related vector field $\tilde{n}$ to $\tilde{N}$. We see that $\tilde{n}$ on $V\cap L$ is inward-pointing. Applying the Boundary Flowout Theorem (Theorem \ref{Boundary Flowout Theorem}) to $\tilde{n}$, write $U_1$ for the obtained neighbourhood of $\partial(V \cap L) \times \{0\} = \partial L \times \{0\}$ and the obtained embedding $\tilde{\rho}_1 : U_1 \to V\cap L$. By existence and uniqueness of integral curves, we have that $h_1 \circ \tilde{\rho}_1 = \sigma \circ k|_{U_1}$. With the embedding $\rho_1 = \tilde{\rho}_1|^L$, we have $l \circ \rho_1 = \Sigma \circ k|_{U_1}$. 

Now, the open submanifold $V \cap \interior L$ of $L$ is such that, the connected components of $V\cap L$ are maximal integral manifolds of the distribution $i^*D$ on the manifold $V \cap \interior M$ where $i : V \cap \interior M \subset M$. Let $\tilde{N}'$ be the $i'$-related vector field to $\tilde{N}$ where $i' : V \cap \interior M \subset V$. Then, $\tilde{N}'$ is a $i^*D$-section. Hence, from Proposition \ref{max contain int curves}, the integral curves of $\tilde{N}'$ which intersect $V \cap \interior L$ are contained in $V \cap \interior L$. Hence, the equation $l \circ \rho_1 = \Sigma \circ k|_{U_1}$ implies that, with $U_2 = \partial L \times (0,1)$, there exists a map $\tilde{\rho}_2 : U_2 \to V \cap \interior L$ such that $h_2 \circ \tilde{\rho}_2 = \tilde{\Sigma} \circ k_2$ where $\tilde{\Sigma}  = \Sigma|_{\partial M \times (0,1)}^{\interior M}$ is an embedding, $k_2 : \partial L \times (0,1) \subset \partial M \times (0,1)$, $h_2 : V \cap \interior L \subset \interior M$. 

Since $V \cap \interior L$ is an open submanifold of the integral manifold $\interior L$ of $\imath^*D$, we have by Proposition \ref{boundarylessweakembedded}, that $V \cap \interior L$ is weakly embedded in $\interior M$. Hence, $\tilde{\rho}_2$ is smooth. Differentiating the equation $h_2 \circ \tilde{\rho}_2 = \tilde{\Sigma} \circ k_2$, we obtain by the Inverse Function Theorem that $\tilde{\rho}_2$ is a local diffeomorphism and hence an embedding (because of injectivity). Hence the map $\rho_2 = \tilde{\rho_2}|^{L}$ is also an embedding and satisfies $l \circ \rho_2 = \Sigma \circ k|_{U_2}$.

We also see that $\rho_1$ and $\rho_2$ agree on $U_1 \cap U_2$ so there exists a (unique) smooth map $\rho : \partial L \times [0,1) \to L$ with $\rho_{|U_1} =\rho_1$ and $\rho_{|U_2} = \rho_2$. It follows that $\rho$ satisfies $l \circ \rho = \Sigma \circ k$ and that $\rho$ is an embedding.
\end{proof}

\begin{lemma}\label{collarexist}
The collar $\Sigma$ is a $D$-collar.
\end{lemma}

\begin{proof}
Let $(p,t) \in \partial M \times [0,1)$. Then, take a neat integral manifold with boundary $L$ with $p \in L$ as in Lemma \ref{intecollar} and inherit the notation from there. Since $\rho$ is an embedding, the relation $l\circ \rho = \Sigma \circ k$ gives us,
\begin{equation*}
\begin{split}
D_{\Sigma(p,t)} &= D_{\rho(p,t)}\\
&= Tl|_{\rho(p,t)}(T_{\rho(p,t)}L)\\
&= Tl|_{\rho(p,t)}(T\rho|_{(p,t)}(T_{(p,t)}(\partial L \times [0,1))))\\
&= T\Sigma|_{(p,t)}(Tk|_{(p,t)}(T_{(p,t)}(\partial L \times [0,1)))).
\end{split}
\end{equation*}

Now, considering the projections $\pi_1,\pi_2$ of $\partial M \times [0,1)$ (onto the indicated factors) and likewise $\tau_1,\tau_2$ on $\partial L \times [0,1)$, with $j : \partial L \subset \partial M$, we obtain the following linear isomorphisms and relations. 
\begin{gather*}
\left(T\pi_1|_{(p,t)}, T\pi_2|_{(p,t)}\right) : T_{(p,t)}(\partial M \times [0,1)) \to T_p\partial M \times T_t[0,1)\\
\left(T\tau_1|_{(p,t)}, T\tau_2|_{(p,t)}\right) : T_{(p,t)}(\partial L \times [0,1)) \to T_p\partial L \times T_t[0,1)\\
\pi_1 \circ k = j \circ \tau_1\\
\pi_2 \circ k = \tau_2.
\end{gather*}

With this, applying Lemma \ref{boundaryint} and a similar computation to the above, we get, 
\begin{gather*}
T\pi_1|_{(p,t)}(Tk|_{(p,t)}(T_{(p,t)}(\partial L \times [0,1)))) = Tj|_p(T_p\partial L) = (\jmath^*D)_p\\
T\pi_2|_{(p,t)}(Tk|_{(p,t)}(T_{(p,t)}(\partial L \times [0,1)))) = T_t[0,1).
\end{gather*}

Using again the fact that $\left(T\pi_1|_{(p,t)}, T\pi_2|_{(p,t)}\right)$ is a linear isomorphism, we see that for any subset  $A \subset T_{(p,t)}(\partial M \times [0,1))$,
\begin{equation*}
A = T\text{Id}(\cdot,t)|_p(T\pi_1|_{(p,t)}(A)) + T\text{Id}(p,\cdot)|_t(T\pi_2|_{(p,t)}(A)).
\end{equation*}

Then, from the above computations, it easily follows that,
\begin{equation*}
T\Sigma(\cdot,t)|_p((\jmath^*D)_p) + T\Sigma(p,\cdot )|_t(T_t[0,1)) = D_{\Sigma(p,t)}.
\end{equation*}
\end{proof}

We summarise our observations.

\begin{corollary}\label{S1 implies S3 in Normal Stefan-Sussmann}
Let $M$ be a manifold with boundary and $D$ be a normal distribution of rank $r$ on $M$. Then we have the implication S1 $\Rightarrow$ S3 in Theorem \ref{Normal Stefan-Sussmann}.
\end{corollary}

\subsection{Implication S2 $\Rightarrow$ S1}

For this implication, we require analysing the product $A \times [0,1)$ where $A$ is a manifold. The domain of a collar is what this analysis will be applied to.

\subsubsection{Product analysis}

Let $A$ be a smooth manifold and $D$ be a distribution of rank $r$ on $A \times B$, where either $B = [0,1)$ or $B = (0,1)$. Set $M = A \times B$ and $n = \dim M$. Denote by $\pi_1 : M \to A$ and $\pi_2 : M \to B$ the projections. 

The following argument uses many elements from the proof of Lemma 19.5 in Lee's book \cite{Lee}.

\begin{lemma}
Assume $B = [0,1)$. Let $p \in A$ such that $D$ has constant rank along the curve $\gamma_p = \text{\emph{Id}}(p,\cdot) : [0,1) \to M$. Then, setting $R = r(p,0)$, there exists a neighbourhood $U$ of $(p,0)$ in $M$ and local 1-forms $\omega_{R+1},...,\omega_{n} : U \to T^*M$ and an $\epsilon > 0$ such that $\gamma_p([0,\epsilon)) \subset U$ for which $D_{\gamma_p(t)} = \bigcap_{i = 1}^n\ker \omega_{R+i}|_{\gamma_p(t)}$ for $t \in [0,\epsilon)$.
\end{lemma}

\begin{proof}
Let $v_1,...,v_R$ be a basis for $D_{(p,0)}$. Then, extend each $v_i$ to a $D$-section $X_i$. Then, the subset $V \subset M$ for which $X_1,...,X_R$ are linearly independent is an open neighbourhood of $(p,0)$. So, $X_1|_{V},...,X_R|_{V}$ are linearly independent local vector fields on $M$. Hence, in a neighbourhood $U$ of $(p,0)$, there exist local vector fields $Y_{R+1},...,Y_n : U \to TM$ such that $(X_1|_{U},...,X_R|_{U},Y_{R+1},...,Y_n)$ forms a local frame for $TM$ on $U$. Then, consider the local co-frame $(\omega_1,...\omega_R,\omega_{R+1},...,\omega_n)$ of $T^*M$ on $U$ dual to $(X_1|_{U},...,X_R|_{U},Y_{R+1},...,Y_n)$. 

For $q \in U$, we get,
\begin{equation*}
\bigcap_{i = 1}^n\ker \omega_{R+i}|_q  = \text{span}(X_1|_q,...,X_R|_q) \subset D_q.
\end{equation*}

Since $U$ is a neighbourhood of $(p,0)$, there is an $\epsilon > 0$ with $\gamma_p([0,\epsilon)) \subset U$. For $t \in [0,\epsilon)$, we hence get, $\bigcap_{i = 1}^n\ker \omega_{R+i}|_{\gamma_p(t)} \subset D_{\gamma_p(t)}$. Moreover, $\bigcap_{i = 1}^n\ker \omega_{R+i}|_{\gamma_p(t)}$ and $D_{\gamma_p(t)}$ both have dimension $R$. Hence, we must have that $D_{\gamma_p(t)} = \bigcap_{i = 1}^n\ker \omega_{R+i}|_{\gamma_p(t)}$.
\end{proof}

Hence, we get the following uniqueness result.

\begin{corollary}\label{interior determines dist}
Assume $B = [0,1)$. Assume that $E$ is a distribution on $M$ such that $D_{(p,t)} = E_{(p,t)}$ for $(p,t) \in \interior M$ and for each for $p \in A$, $D$ and $E$ have constant rank along the curves $\gamma_p = \text{\emph{Id}}(p,\cdot) : [0,1) \to M$ for $p \in A$. Then, $D = E$.
\end{corollary}

For what follows, denote $\iota_t = \text{Id}(\cdot,t) : A \to M$ for $t \in B$ and recall that,
\begin{equation*}
(T\pi_1|_{(p,t)},T\pi_2|_{(p,t)}) : T_{(p,t)}(A \times B) \to T_{p}A \times T_{t}B,    
\end{equation*}
is a linear isomorphism.

\begin{lemma}\label{pullback dist is dist}
Assume $B = [0,1)$. Let $D_1$ be a distribution on $A$ of rank $r_1$. Then, $E = \pi_1^*D_1$ is a distribution of rank $r_1 \circ \pi_1 + 1$ on $M$. If $D_1$ is integrable, then $E$ is neatly integrable by boundaries.
\end{lemma}

\begin{proof}
For $(p,t) \in M$, we have $E_{(p,t)} = (T\pi_1|_{(p,t)})^{-1}((D_1)_p)$ so that, since $(D_1)_p$ is a vector subspace of $T_pA$, $E_{(p,t)}$ is a vector subspace of $T_{(p,t)}M$. Since $(T\pi_1|_{(p,t)},T\pi_2|_{(p,t)})$ is a linear isormorphism, we have that $T\pi_1|_{(p,t)}(E_{(p,t)}) = (D_1)_p$ and $T\pi_2|_{(p,t)}(E_{(p,t)}) = T_t[0,1)$, so $(T\pi_1|_{(p,t)},T\pi_2|_{(p,t)})(E_{(p,t)}) = (D_1)_p \times T_t[0,1)$ and hence $E_{(p,t)}$ has dimension $\dim{(D_1)_p} + 1 = (r_1\circ \pi_1 + 1)(p)$.

Now, let $u \in E_{(p,t)}$. Since $u \in T_{(p,t)}M$, setting $v = T\pi_1|_{(p,t)}(u)$, we have, for some $\alpha \in \mathbb{R}$,
\begin{equation*}
u = T\iota_t|_p(v) + \alpha \frac{\partial}{\partial t}|_{(p,t)} = T\pi_1|_{(p,t)}(u) + \alpha \frac{\partial}{\partial t}|_{(p,t)},    
\end{equation*}
since $(T\pi_1|_{(p,t)},T\pi_2|_{(p,t)})$ is a linear isomorphism, $T_t[0,1)$ is 1-dimensional and $T\pi_2|_{(p,t)}\left(\frac{\partial}{\partial t}|_{(p,t)} \right) = 0$.

Moreover, $v = T\pi_1|_{(p,t)}(u) \in (D_1)_p$. Hence, there is a $D_1$-section $W$ with $W|_p = v$. Then, the map $X : M \to TM$ defined by $X(p,t) = T\iota_t|_p(W_p)$, is a vector field on $M$. Then, consider the vector field $Y = X + \alpha \frac{\partial}{\partial t}$. 

Clearly $Y|_{(p,t)} = u$ and for $(p',t') \in M$, $T\pi_1|_{(p',t')}(X|_{(p',t')}) = V_{p'} \in (D_1)_{p'}$ and hence $X|_{(p',t')} \in (\pi_1^*D_1)_{(p',t')}$. So, $Y$ is an $E$-section with $Y|_{(p,t)} = u$.

Now assume that $D_1$ is integrable. Then, let $(p,t) \in M$. Then, take a integral manifold $I$ of $D_1$ containing $p$. Then, $I \times [0,1)$ is a connected submanifold of $M$ with boundary whereby $\partial(I \times [0,1)) \subset \partial M$.

Now, let $(p',t') \in I \times [0,1)$. Denoting by $\tau_1 : I \times [0,1) \to I$ the projection onto the first factor, we have $\iota \circ \tau_1 = \pi_1 \circ \iota'$ where $\iota : I \subset A$ and $\iota' : I \times [0,1) \subset M$ denote inclusions. Combined with the fact that $I$ is an integral manifold of $D_1$, we get,
\begin{equation*}
T\pi_1|_{(p',t')}(T\iota'|_{p',t')}(T_{(p',t')}(I \times [0,1)))) = T\iota|_{p'}(T_{p'}I) = (D_1)_{p'}.
\end{equation*}

Thus, $T\iota'(T_{(p',t')}(I \times [0,1))) \subset (T\pi_1|_{(p,t)})^{-1}((D_1)_{p'}) = E_{(p',t')}$. Moreover, since $(T\iota'|_{p',t')}$ is an immersion,
\begin{equation*}
\dim{T\iota'|_{p',t')}(T_{(p',t')}(I \times [0,1))} = \dim{T_{(p',t')}(I \times [0,1)} = \dim{I} + 1 = \dim{(D_1)_{p'}} + 1 = \dim E_{(p',t')}.
\end{equation*}

Hence, we must have $T\iota'(T_{(p',t')}(I \times [0,1))) = (T\pi_1|_{(p,t)})^{-1}((D_1)_{p'}) = E_{(p',t')}$. Hence, $I \times [0,1)$ is a neat integral manifold with boundary containing $(p,t)$.
\end{proof}

\begin{lemma}\label{exists pullback dist}
Assume $B =  (0,1)$. Assume $D$ is integrable and that $\frac{\partial}{\partial t}$ is a $D$-section. Then, there exists an integrable distribution $D_1$ on $A$ such that $D = \pi_1^*D_1$.
\end{lemma}

\begin{proof}
Consider, for example, the embedding $\iota_{1/2} = \text{Id}(\cdot,1/2) : A \to M$ and the pull-back $D_1 = \iota_{1/2}^*D$. Indeed, with $f = \pi_2|^{\mathbb{R}}$, we get the 1-form $\omega = df$ on $M$. Moreover, we have that $\iota_{1/2}^*\omega = 0$ and $\omega\left( \frac{\partial}{\partial t}\right) = 1$. Hence, by Proposition \ref{pullback dist to codim 1}, we have that $D_1$ is a distribution on $A$ of rank $r\circ \iota_{1/2} -1$.

Now, let $I$ be a maximal integral manifold of $D$. Then, let $p \in \pi_1(I)$ so that $(p,t) \in I$ for some $t \in (0,1)$. Then, $\gamma_p  = \text{Id}(p,\cdot)$ is an integral curve of $\frac{\partial}{\partial t}$ whereby $\gamma_p \cap I \neq \emptyset$. Hence, by Proposition \ref{max contain int curves}, $\gamma_p \subset I$. Hence, $\{p\} \times (0,1) \subset I$. Hence, $\text{Pts}(I) = \pi_1(I) \times (0,1)$. Since $\pi_2|_I$ is a submersion, by the Preimage Theorem, there exists a smooth structure on $\pi_1(I)$, which we denote by $I_1$, such that $I_1$ is a connected submanifold of $A$ whereby $I = I_1 \times (0,1)$. From this, it follows that $D_1$ is integrable.
\end{proof}

\begin{lemma}\label{finale for product analysis}
Assume $B = [0,1)$. Assume that $\imath^*D$ is integrable (where $\imath : \interior M \subset M$ is the inclusion), $\frac{\partial}{\partial t}$ is a $D$-section and that $D$ is of constant rank along the curves $\gamma_p = \text{Id}(p,\cdot) : [0,1) \to N \times [0,1)$ for $p \in N$. Then, $D$ is neatly integrable by boundaries.
\end{lemma}

\begin{proof}
First, since $\imath^*D$ is integrable and $\frac{\partial}{\partial t}$ is a $D$-section, it follows from Lemma \ref{exists pullback dist} that there exists an integrable distribution $D_1$ on $N$ such that $\imath^*D = \tilde{\pi}_1^*D_1 = (\pi_1\circ \imath)^*D_1$, where $\tilde{\pi}_1 : \interior M \to A \times (0,1)$ is projection onto the first factor. 

Then, consider $E = \pi_1^*(D_1)$. We have from Lemma \ref{pullback dist is dist} that $E$ is a neatly integrable distribution of rank $r_1 \circ \pi_1 + 1$ (where $r_1$ is the rank of $D_1$). Hence, $E$ is of constant rank along the curves $\gamma_p = \text{Id}(p,\cdot) : [0,1) \to M$. Moreover, we see that $\imath^*E = \imath^*(\pi_1^*D_1) = (\pi_1\circ \imath)^*D_1 = \imath^*D$ and hence $D_{(p,t)} = E_{(p,t)}$ for all $(p,t) \in \interior M$.

Hence, by Corollary \ref{interior determines dist}, we must have that $D=E$ and hence, $D$ is neatly integrable by boundaries.
\end{proof}

With this, we derive implication $S2 \Rightarrow S1$ in Theorem \ref{Normal Stefan-Sussmann}.

\subsubsection{Application of the product analysis}

\begin{lemma}\label{S2 implies S1 in Normal Stefan-Sussmann}
Let $M$ be a manifold with boundary $\partial M \neq \emptyset$ and $D$ be a normal distribution of rank $r$ on $M$. Let $\imath : \interior M \subset M$ and $\jmath : \partial M \subset M$ denote the inclusions. Then we have the implication $S2 \Rightarrow S1$ in Theorem \ref{Normal Stefan-Sussmann}.
\end{lemma}

\begin{proof}
Assume that $\imath^*D$ is integrable and $D$ is precollared.

Let $\Sigma : \partial M \times [0,1) \to M$ be a $D$-precollar. Denote by $\sigma = \Sigma|^V$ the co-restriction diffeomorphism onto $V = \Sigma(\partial M \times [0,1))$. The inclusions we require are denoted by $\iota : V \subset M$, $\tilde{\imath} : \interior V \subset V$, $\tilde{\jmath} : \partial V \subset V$ and $\imath' : \partial M \times (0,1) \subset \partial M \times [0,1)$. Lastly, we set $\tilde{D} = \tilde{\imath}^*D$.

First, notice that since $\tilde{D} = \tilde{\imath}^*D$ is a distribution, it is clear that then $D' = \sigma^*\tilde{D}$ is a distribution on $\partial M \times [0,1)$ since $\sigma$ is a diffeomorphism. Moreover, since $\imath^*D$ is integable, it follows easily that $\tilde{\imath}^*\tilde{D}$ is integrable. Since again $\sigma$ is a diffeomorphism, it is clear then that $\iota'^*D'$ is integrable.

Now, since $\Sigma$ is a $D$-precollar, the vector field $\tilde{N} = T\sigma \circ \frac{\partial }{\partial t} \circ \sigma^{-1}$ is a $\tilde{D}$-section and $\tilde{D}$ has constant rank along the curves $\sigma(p,\cdot)$ for $p \in \partial V = \partial M$. Hence, $\frac{\partial}{\partial t}$ is a $D'$-section and $D'$ has constant rank along the curves $\gamma_p = \text{Id}(p,\cdot)$ for $p \in \partial M$.

Hence, altogether, by Lemma \ref{finale for product analysis}, $D'$ is neatly integrable by boundaries. Hence, since $\sigma$ is a diffeomorphism, it easily follows that $\tilde{D}$ is neatly integrable by boundaries. From this, since again $\imath^*D$ is integrable by assumption, we obtain altogether that $D$ is neatly integrable by boundaries.
\end{proof}

\subsection{Implications S3 $\Rightarrow$ S2, S3 $\Rightarrow$ S4 and S4 $\Rightarrow$ S1}

Let $M$ be a manifold with boundary $\partial M \neq \emptyset$ and $D$ be a normal distribution of rank $r$ on $M$. Let $\imath : \interior M \subset M$ and $\jmath : \partial M \subset M$ denote the inclusions.

\begin{remark}
Note that collared immediately implies pre-collared. That is, the desired implication $S3 \Rightarrow S4$ is trivial.
\end{remark}

\begin{lemma}
Assume that $D$ is collared and $\imath^*D$ and $\jmath^*D$ are integrable. For all $p \in M$, then, there exists a chart $(U,\varphi)$ neatly adapted to $D$ at $p$.
\end{lemma}

\begin{proof}
Let $\Sigma : \partial M \times [0,1) \to M$ be a $D$-collar. Suppose $p \in \partial M$. Since $\jmath^*D$ is integrable, by Theorem \ref{Stefan-Sussmann}, there exists a chart $(U',\tilde{\varphi})$ on $\partial M$ adapted to $\jmath^*D$ at $p$. We may choose $(U',\tilde{\varphi})$ so that $\tilde{\varphi}(U') = (-\epsilon,\epsilon)^{n-1}$ for some $0 < \epsilon \leq 1$. Then, $(\varphi'^1,...,\varphi'^{n-1}) = (\tilde{\varphi}^{r(p)},...,\tilde{\varphi}^n,\tilde{\varphi}^1,...,\tilde{\varphi}^{r(p)-1})$ defines a new chart $(U',\varphi')$. Then, with the open set $V' = U'\times[0,\epsilon)$, setting $U = \Sigma(V')$ gives the chart $(U,\varphi)$ where $\varphi = (\varphi' \times \text{id}) \circ \sigma'^{-1}$ where we have set $\sigma' = \Sigma|_{V'}^{U}$ and $\text{id} : [0,\epsilon) \to [0,\epsilon)$. Due to the fact that $\Sigma$ is a $D$-collar, it is easy to see that $(U,\varphi)$ is neatly adapted to $D$.

Suppose instead that $p \in \interior M$. Then, from Theorem \ref{Stefan-Sussmann}, since $\imath^*D$ is integrable, there exists a chart $(U,\tilde{\varphi})$ of $\interior M$ (and hence $M$) adapted to $\imath^*D$. Then, $(\varphi^1,...,\varphi^n) = (\varphi'^{r(p)+1},...,\varphi'^{n},\varphi'^1,...,\varphi'^{r(p)})$ defines a new chart $(U,\varphi)$. It is easy to see that $(U,\varphi)$ is neatly adapted to $D$.
\end{proof}

The final implication is essentially the same as the proof presented in Ruldolf and Schmidt's book \cite{RudolphSchmidt} for the regular Stefan-Sussmann Theorem. 

\begin{lemma}
Assume that at every point $p \in M$, there exists a chart $(U,\varphi)$ neatly adapted to $D$ at $p$. Then $D$ is neatly integrable by boundaries.
\end{lemma}

\begin{proof}
Set $n = \dim M$. Let $p \in M$. Then, take a chart $(U,\varphi)$ neatly adapted to $D$ at $p$. The case $r(p) = n$ is trivial. Now suppose $r(p) < n$.

We have for some $\epsilon > 0$ that either $\varphi(U) = (-\epsilon,\epsilon)^n$ or $\varphi(U) = (-\epsilon,\epsilon)^n \cap \mathbb{H}^n$. Write $S$ for the structure on $\varphi^{-1}(\{0\} \times (-\epsilon,\epsilon)^{r(p)})$ as an embedded submanifold with boundary in $\mathbb{R}^n$ where $\partial S = S \cap \mathbb{H}^n$. Then, $L = (\varphi^{-1})(S)$ is an embedded submanifold with boundary in $M$ with $\partial L \subset \partial M$. Now, since $D$ is of constant rank on $L$ and $\partial_{n-r(p)+1},...,\partial_n$ are local $D$-sections, it follows easily that $L$ is an integral manifold with boundary of $D$. Altogether, $L$ is a neat integral manifold with boundary of $D$ where $p \in L$.
\end{proof}

We summarise the observations of this section.

\begin{corollary}\label{S3 implies S2, S3 implies S4 and S4 implies S1 in Normal Stefan-Sussmann}
Let $M$ be a manifold with boundary and $D$ be a normal distribution of rank $r$ on $M$. Then we have the implications $S3 \Rightarrow S2$, $S3 \Rightarrow S4$ and $S4 \Rightarrow S1$ in Theorem \ref{Normal Stefan-Sussmann}.
\end{corollary}

\subsection{Mutual equivalence of S1, S2, S3 and S4}

We need only to verify that we have collected sufficient implications for mutual equivalence.

\begin{proof}[Proof of Theorem \ref{Normal Stefan-Sussmann}] 
As mentioned, the required opening statement is Corollary \ref{opening of Normal Steffan-Sussmann}. In the following diagram, we label indicated implication by the result which established it.

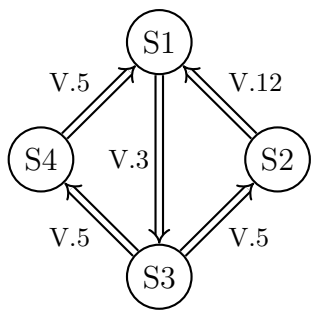
\begin{center}
\begin{tikzpicture}[node distance=2.2cm,thick,
main node/.style={circle,draw,font=\large},
line/.style={double distance=2pt, -Implies,thick}
]
  \node[main node] (S1) {S1};
  \node[main node] (S2) [below right of=S1] {S2};
  \node[main node] (S3) [below left of=S2] {S3};
  \node[main node] (S4) [below left of=S1] {S4};

  \draw[line] (S1) to node[left] {\ref{S1 implies S3 in Normal Stefan-Sussmann}} (S3);
  \draw[line] (S3) to node[below right] {\ref{S3 implies S2, S3 implies S4 and S4 implies S1 in Normal Stefan-Sussmann}} (S2);
  \draw[line] (S2) to node[above right] {\ref{S2 implies S1 in Normal Stefan-Sussmann}} (S1);
  \draw[line] (S3) to node[below left] {\ref{S3 implies S2, S3 implies S4 and S4 implies S1 in Normal Stefan-Sussmann}} (S4);
  \draw[line] (S4) to node[above left] {\ref{S3 implies S2, S3 implies S4 and S4 implies S1 in Normal Stefan-Sussmann}} (S1);
\end{tikzpicture}
\end{center}

From this it is clear that S1, S2, S3 and S4 are all equivalent.
\end{proof}

\section{Acknowledgements}

The first author would like to acknowledge his supervisors and co-authors David Pfefferl{\'e} and Luchezar Stoyanov for the valuable and frequent discussions which occurred while this paper was being written. This paper was written while the first author received an Australian Government Research Training Program Scholarship at The University of Western Australia.

\bibliography{DistributionReferences}

\end{document}